\documentclass[12pt,reqno]{amsart}
\usepackage{enumerate, latexsym, amsmath, amsfonts, amssymb, amsthm, graphicx, color}
\usepackage{hyperref, url}
\def\pmod #1{\ ({\rm{mod}}\ #1)}
\def\Z{\Bbb Z}
\def\N{\Bbb N}

\def\l{\left}
\def\r{\right}
\def\bg{\bigg}
\def\({\bg(}
\def\){\bg)}
\def\t{\text}
\def\f{\frac}

\def\ls{\leqslant}
\def\gs{\geqslant}

\def\bi{\binom}

\def\eq{\equiv}

\def\Proof{\noindent{\it Proof}}
\def\Ack{\medskip\noindent {\bf Acknowledgment}}
\theoremstyle{plain}
\newtheorem{theorem}{Theorem}[section]
\newtheorem{lemma}[theorem]{Lemma}

\theoremstyle{remark}
\newtheorem{remark}[theorem]{Remark}

\makeatletter
\@namedef{subjclassname@2010}{%
  \textup{2010} Mathematics Subject Classification}
\makeatother
 \vspace{4mm}

\begin{document}
 \baselineskip=17pt
\hbox{Int. J. Number Theory 16(2020), no.\,5, 981--1003.}
\medskip
\title[Congruences for Ap\'ery numbers $\beta_{n}=\sum_{k=0}^{n}\binom{n}{k}^2\binom{n+k}{k}$]
{Congruences for Ap\'ery numbers $\beta_{n}=\sum_{k=0}^{n}\binom{n}{k}^2\binom{n+k}{k}$}
\author[H.-Q. Cao, Yu. Matiyasevich and Z.-W. Sun]{Hui-Qin Cao, Yuri Matiyasevich, and Zhi-Wei Sun}

\thanks{
\newline \indent The research is supported by the NSFC-RFBR Cooperation and Exchange Program (NSFC Grant No. 11811530072,
RFBR Grant No. 18-51-53020-GFEN\underline{\phantom{x}}a).
The first and the third authors are also supported by the Natural Science Foundation of China (Grant No. 11971222).}

\address {(Hui-Qin Cao) Department of Applied Mathematics, Nanjing Audit University,
Nanjing 211815, People's Republic of China}
\email{\url{caohq@nau.edu.cn}}

\address {(Yuri Matiyasevich) St. Petersburg Department of Steklov Mathematical Institute of Russian Academy of Sciences, Fontanka 27, 191023, St. Petersburg, Russia}
\email{\url{yumat@pdmi.ras.ru}}

\address {(Zhi-Wei Sun, corresponding author) Department of Mathematics, Nanjing
University, Nanjing 210093, People's Republic of China}
\email{\url{zwsun@nju.edu.cn}}

\begin{abstract}
In this paper we establish some congruences involving the Ap\'ery numbers
$\beta_{n}=\sum_{k=0}^{n}\binom{n}{k}^2\binom{n+k}{k}$.
For example, we show that $$\sum_{k=0}^{n-1}(11k^2+13k+4)\beta_k\equiv0\pmod{2n^2}$$
for any positive integer $n$, and
$$\sum_{k=0}^{p-1}(11k^2+13k+4)\beta_k\equiv 4p^2+4p^7B_{p-5}\pmod{p^8}$$
for any prime $p>3$, where $B_{p-5}$ is the $(p-5)$th Bernoulli number. We also present certain relations between congruence properties of
the two kinds of A\'pery numbers,
 $\beta_n$ and $A_n=\sum_{k=0}^n\bi nk^2\bi{n+k}k^2$.
\end{abstract}

\maketitle
\section{Introduction}
\setcounter{equation}{0}

In 1979 R. Ap\'ery (see \cite{A} and \cite{Po}) established the irrationality of $\zeta(3)=\sum_{n=1}^\infty 1/n^3$ by using the Ap\'ery numbers
$$A_n=\sum_{k=0}^n\bi
nk^2\bi{n+k}k^2\ \ (n\in\N=\{0,1,2,\ldots\}).$$
His method also allowed him to prove the irrationality of $\zeta(2)=\pi^2/6$ via another kind of Ap\'ery numbers
$$\beta_n:=\sum_{k=0}^n\bi nk^2\bi{n+k}k\ \ (n\in\N).$$

In 2012, the third author
\cite{Sun12a} proved that
$$\sum_{k=0}^{n-1}(2k+1)A_k\eq0\pmod n\ \ \ \t{for all}\ n\in\Z^+=\{1,2,3,\ldots\}$$
and
$$\sum_{k=0}^{p-1}(2k+1)A_k\eq p+\f 76p^4B_{p-3}\pmod{p^5}\quad\t{for any prime}\ p>3,$$
where $B_0,B_1,B_2,\ldots$ are the Bernoulli numbers defined by
$$B_0=1, \ \t{and}\ \sum_{k=0}^n\bi{n+1}kB_k=0\ \ \t{for all}\ n\in\Z^+.$$
(For the basic properties of Bernoulli numbers, see \cite[pp. 228-248]{IR}.)
The third author
also conjectured that
$$\sum_{k=0}^{n-1}(2k+1)(-1)^kA_k\eq0\pmod n\quad\t{for all}\ n\in\Z^+$$
and
$$\sum_{k=0}^{p-1}(2k+1)(-1)^kA_k\eq p\l(\f p3\r)\pmod{p^3}\ \ \t{for any prime}\ p>3,$$
which was confirmed by V.J.W. Guo and J. Zeng \cite{GZ}.
In 2016
the third author
\cite{Sun16} showed that
$$\sum_{k=0}^{n-1}(6k^3+9k^2+5k+1)(-1)^kA_k\eq0\pmod {n^3}\quad\t{for all}\ n\in\Z^+.$$
In view of this, it is natural to ask whether the Ap\'ery numbers $\beta_n\ $ 
also have such kind of congruence properties. This is the main motivation of this paper.

Another motivation is due to  a theorem
of
E.\,Rowland and R.\,Yassawi  \cite{RY} about $A_n\ (n\in\N)$ modulo $16$.
We will extend their result to $A_n$ modulo $32$;
quite unexpectedly, this requires consideration
of $\beta_n$
 modulo $8$.

Now we state our main results which were first found via {\tt Mathematica}.

\begin{theorem}\label{Th1.1} For any positive integer $n$, we have
\begin{align}\label{1.1}\sum_{k=0}^{n-1}(11k^2+13k+4)\beta_k\eq0\pmod{2n^2},
\\\label{1.2}\sum_{k=0}^{n-1}(11k^2+9k+2)(-1)^k\beta_k\eq0\pmod{2n^2},
\\\label{1.3}\sum_{k=0}^{n-1}(11k^3+7k^2-1)(-1)^k\beta_k\eq0\pmod{n^2}.
\end{align}
\end{theorem}

\begin{theorem}\label{main1}
Let $p>3$ be a prime. Then
\begin{align}
\label{1.4}\sum_{k=0}^{p-1}(11k^2+13k+4)\beta_k&\equiv 4p^2+4p^7B_{p-5}\pmod {p^8},\\
\label{1.5}\sum_{k=0}^{p-1}(11k^2+9k+2)(-1)^{k}\beta_k&\equiv 2p^2+10p^3H_{p-1}-p^7B_{p-5}\pmod{p^8},
\end{align}
and
\begin{equation}\label{1.6}\begin{aligned}
&\sum_{k=0}^{p-1}(11k^3+7k^2-1)(-1)^{k}\beta_k
\\\equiv& 2p^3-3p^2+(10p-5)p^3H_{p-1}-\f32p^7B_{p-5}\pmod{p^8},
\end{aligned}\end{equation}
where $H_{p-1}$ denotes the harmonic number $\sum_{k=1}^{p-1}\f1k$.
\end{theorem}

\begin{theorem}\label{a16b4}
Let $n$ be a nonnegative integer. If
\begin{equation}
A_n\equiv 8j+4i+1 \pmod{16}
\end{equation}
with $\{i,j\}\subseteq \{0,1\}$,
then
\begin{eqnarray}
  \beta_n\equiv 2(i+j)+1 \pmod{4}.
\end{eqnarray}\label{ab}
\end{theorem}
\begin{remark} For any $n\in\N$, we clearly have
$$A_n=1+4\sum_{0<k\ls n}\bi{n+k}{2k}^2\bi{2k-1}{k-1}^2\eq1\pmod 4.$$
\end{remark}

\begin{theorem}\label{a32b8} Let $n$ be a nonnegative integer.
If
$$
A_n\equiv 16k+8j+4i+1 \pmod{32}\ \t{and}\ \beta_n\equiv4m+2l+1 \pmod{8}
$$
with $\{i,j,k,l,m\}\subseteq \{0,1\}$,
then
\begin{equation}
L_2(n)\equiv 4(k+m)+2j+i \pmod{8},
\label{ab32}
\end{equation}
where $L_2(n)$ (the
binary run-length of $n$) is defined as the number
of blocks of consecutive $0$'s and $1$'s in the binary expansion of $n$.
\end{theorem}

We will prove Theorem \ref{Th1.1} in the next section.
We are going to provide some lemmas in Section 3 and show Theorem \ref{main1} in Section 4.
 Theorems \ref{a16b4} and \ref{a32b8}
will be proved in Section \ref{bandA}.

\section{Proof of Theorem 1.1}
\setcounter{equation}{0}

\begin{lemma}\label{Lem2.1} For any $n\in\Z^+$, we have
\begin{align}
\label{2.1}\sum_{k=0}^{n-1}(11k^2+13k+4)\beta_k=-\sum_{k=0}^{n-1}\binom{n}{k+1}^2\binom{n+k}{k}a_1(n,k),\\
\label{2.2}\sum_{k=0}^{n-1}(11k^2+9k+2)(-1)^{n-1-k}\beta_k=\sum_{k=0}^{n-1}\binom{n}{k+1}^2\binom{n+k}{k}a_2(n,k),\\
\label{2.3}\sum_{k=0}^{n-1}(11k^3+7k^2-1)(-1)^{n-1-k}\beta_k
=\sum_{k=0}^{n-1}\binom{n}{k+1}^2\binom{n+k}{k}a_3(n,k),
\end{align}
where
\begin{equation}\label{2.4}\begin{aligned}
a_1(n,k)=&n^3-n^2k-nk^2-2n^2-5nk+k^2-4n+2k+1,\\
a_2(n,k)=&n(3n-k-1),\\
a_3(n,k)=&4n^3-2n^2k-nk^2-4n^2-3nk+k^2-2n+2k+1.
\end{aligned}\end{equation}
\end{lemma}
\Proof.  Note that
$$\sum_{k=0}^{n-1}(11k^2+13k+4)\beta_k=\sum_{k=0}^{n-1}\sum_{l=0}^{n-1}F(k,l)$$
is a double sum with
$$F(k,l)=(11k^2+13k+4)\bi kl^2\bi{k+l}l.$$
Using the method in Mu and Sun's paper \cite{MS}, we find
$$F_1(k,l)=-\f{(k-l)^2a_1(k,l)}{(l+1)^2(11k^2+13k+4)}F(k,l)$$
and
$$F_2(k,l)=\f{l(3kl-l^2+2k-2l)}{11k^2+13k+4}F(k,l)$$
so that
$$F(k,l)=(F_1(k+1,l)-F_1(k,l))+(F_2(k,l+1)-F_2(k,l))$$
which can be verified directly.
This allows us to reduce the double sum $\sum_{k=0}^{n-1}\sum_{l=0}^{n-1}F(k,l)$
to the right-hand side of \eqref{2.1}.

Identities
\eqref{2.2} and \eqref{2.3} can be deduced similarly; in fact,
\begin{align*}&(11k^2+9k+2)(-1)^{k}\bi kl^2\bi{k+l}l
\\=&(G_1(k+1,l)-G_1(k,l))+(G_2(k,l+1)-G_2(k,l))
\end{align*}
and
\begin{align*}&(11k^3+7k^2-1)(-1)^{k}\bi kl^2\bi{k+l}l
\\=&(H_1(k+1,l)-H_1(k,l))+(H_2(k,l+1)-H_2(k,l)),
\end{align*}
where
\begin{align*}G_1(k,l)=&-\f{(k-l)^2a_2(k,l)}{(l+1)^2}(-1)^k\bi kl^2\bi{k+l}l,
\\ G_2(k,l)=&-l(6k+l+3)(-1)^k\bi kl^2\bi{k+l}l,
\\H_1(k,l)=&-\f{(k-l)^2a_3(k,l)}{(l+1)^2}(-1)^k\bi kl^2\bi{k+l}l,
\\H_2(k,l)=&-l(8k^2+l^2+4k+2l+2)(-1)^k\bi kl^2\bi{k+l}l.
\end{align*}
This ends our proof.
\qed

\begin{lemma}\label{Lem2.2} For any positive even integer $n$, we have
\begin{equation}\label{2.5}\sum_{k=0}^{n-1}\bi{n-1}k\bi{n+k}{2k}\bi{2k}{k+1}\eq1\pmod2.
\end{equation}
\end{lemma}
\Proof. Clearly,
$$\sum_{k=0}^{n-1}\bi{n-1}k\bi{n+k}{2k}\bi{2k}{k+1}
=\sum_{0<k<n}\bi{n-1}k\bi{n+k}{k+1}\bi{n-1}{k-1}.$$
For $k\in\{0,\ldots,n-1\}$, if $2\mid k$ then $n+k$ is even and hence
$$\bi{n+k}{k+1}=\f{n+k}{k+1}\bi{n+k-1}k\eq0\pmod2;$$
if $2\nmid k$ then
$$\bi{n-1}k+\bi{n-1}{k-1}=\bi nk=\f nk\bi{n-1}{k-1}\eq0\pmod2.$$
Therefore,
\begin{align*}\sum_{k=0}^{n-1}\bi{n-1}k\bi{n+k}{2k}\bi{2k}{k+1}
\eq&\sum_{k=0}^{n-1}\bi{n-1}k\bi{n+k}{k+1}\bi{n-1}k
\\\eq&\sum_{k=0}^{n-1}\bi{n-1}k(-1)^{k+1}\bi{n+k}{k+1}
\\=&(-1)^n\pmod 2,
\end{align*}
where we have applied (in the last step) the Chu-Vandermonde identity (cf. \cite[(3.1)]{Gould}). This proves \eqref{2.5}. \qed

\medskip
\noindent{\it Proof of Theorem 1.1}. For each $k=0,\ldots,n-1$, clearly
$$\bi n{k+1}^2(k+1)^2=n^2\bi{n-1}k^2$$
and $$\bi n{k+1}^2\bi{n+k}kn(k+1)=n^2\bi{n-1}k\bi{n+k}{2k}(n-k)C_k,$$
where $C_k$ is the Catalan number $\bi{2k}k/(k+1)=\bi{2k}k-\bi{2k}{k+1}$.

(i) As
$$a_1(n,k)=n^2(n-k-2)-n(k+1)(k+4)+(k+1)^2,$$
we have
\begin{align*}&\f1{n^2}\sum_{k=0}^{n-1}\binom{n}{k+1}^2\binom{n+k}{k}a_1(n,k)
\\=&\sum_{k=0}^{n-1}\bi n{k+1}^2\bi{n+k}k(n-k-2)
\\&-\sum_{k=0}^{n-1}\bi{n-1}k\bi{n+k}{2k}(n-k)(k+4)C_k
+\sum_{k=0}^{n-1}\bi{n-1}k^2\bi{n+k}k.
\end{align*}

Clearly $11k^2+13k+4\eq0\pmod2$ for all $k=0,\ldots,n-1$. So, by Lemma 2.1 and the above, (1.1) holds when $2\nmid n$.

Now suppose that $n$ is even. By Lemma 2.1 and the above, we have
\begin{align*}&\f1{n^2}\sum_{k=0}^{n-1}(11k^2+13k+4)\beta_k
\\\eq&\sum_{k=0}^{n-1}\bi n{k+1}\bi{n+k}kk+\sum_{k=0}^{n-1}\bi{n-1}k\bi{n+k}{2k}kC_k
\\&+\sum_{k=0}^{n-1}\bi{n-1}k\bi{n+k}k
\pmod2.
\end{align*}
With the help of the Chu--Vandermonde identity,
\begin{align*}\sum_{k=0}^{n-1}\bi n{k+1}\bi{n+k}kk=&\sum_{k=0}^{n-1}\bi{n}{n-1-k}\bi{-n-1}k(-1)^k k
\\\eq&(-n-1)\sum_{k=1}^{n-1}\bi{n}{n-1-k}\bi{-n-2}{k-1}
\\\eq&\bi{-2}{n-2}=(-1)^n(n-1)\eq1\pmod2
\end{align*} and
\begin{align*}\sum_{k=0}^{n-1}\bi{n-1}k\bi{n+k}k
=&\sum_{k=0}^{n-1}\bi{n-1}k\bi{-n-1}k(-1)^k
\\\eq&\sum_{k=0}^{n-1}\bi{n-1}{n-1-k}\bi{-n-1}k
\\=&\bi{-2}{n-1}=(-1)^n n\eq0\pmod2.
\end{align*}
In view of \eqref{2.5}, we also have
\begin{equation}\label{2.7}\sum_{k=0}^{n-1}\bi{n-1}k\bi{n+k}{2k}kC_k\eq1\pmod2
\end{equation} since $kC_k=\bi{2k}{k+1}$. Therefore \eqref{1.1} holds.

(ii) As $a_2(n,k)=3n^2-(k+1)n$, we have
 \begin{align*}&\f1{n^2}\sum_{k=0}^{n-1}\binom{n}{k+1}^2\binom{n+k}{k}a_2(n,k)
\\=&3\sum_{k=0}^{n-1}\bi n{k+1}^2\bi{n+k}k-\sum_{k=0}^{n-1}\bi{n-1}k\bi{n+k}{2k}(n-k)C_k.
\end{align*}
Clearly $11k^2+9k+2\eq0\pmod2$ for all $k=0,\ldots,n-1$. So, by Lemma 2.1 and the above, (1.2) holds when $2\nmid n$.

Now suppose that $n$ is even. By Lemma 2.1 and the above, we have
\begin{align*}&\f1{n^2}\sum_{k=0}^{n-1}(11k^2+9k+2)(-1)^k\beta_k
\\\eq&\sum_{k=0}^{n-1}\bi n{k+1}\bi{n+k}k+\sum_{k=0}^{n-1}\bi{n-1}k\bi{n+k}{2k}kC_k\pmod{2}.
\end{align*}
By the Chu--Vandermonde identity,
\begin{align*}\sum_{k=0}^{n-1}\bi n{k+1}\bi{n+k}k
\eq&\sum_{k=0}^{n-1}\bi n{n-1-k}\bi{-n-1}k
=\bi{-1}{n-1}\eq1\pmod2.
\end{align*}
Combining this with \eqref{2.7} we obtain \eqref{1.2}.

(iii) Since
$$a_3(n,k)=2n^2(2n-k-2)-n(k+1)(k+2)+(k+1)^2,$$
we have
\begin{align*}&\f1{n^2}\sum_{k=0}^{n-1}\binom{n}{k+1}^2\binom{n+k}{k}a_3(n,k)
\\=&2\sum_{k=0}^{n-1}\bi n{k+1}^2\bi{n+k}k(2n-k-2)
\\&-\sum_{k=0}^{n-1}\bi{n-1}k\bi{n+k}{2k}(n-k)(k+2)C_k
+\sum_{k=0}^{n-1}\bi{n-1}k^2\bi{n+k}k
\end{align*}
and hence \eqref{1.3} follows from \eqref{2.3}.

In view of the above, we have completed the proof of Theorem 1.1. \qed

\section{Some lemmas}
\setcounter{equation}{0}

For each $s\in\Z^+$, we define
$$
H_{n,s}:=\sum_{0<k\ls n}\frac{1}{k^s}\qquad \t{for}\ n=0,1,2,\ldots,
$$
and call such numbers {\it harmonic numbers of order $s$}. Those $H_n:=H_{n,1}\ (n\in\N)$ are usually called {\it harmonic numbers}.

\begin{lemma}\label{Lem3.1} For any $n\in\Z^+$, we have
\begin{align}\label{equa1}
\sum_{k=1}^{n}\binom{n}{k}\frac{(-1)^{k-1}}{k}=&H_n,
\\\label{equa2}
\sum_{k=1}^{n}\binom{n}{k}\frac{(-1)^{k-1}}{k}H_k=&H_{n,2},
\\\label{equa3}
\sum_{k=1}^{n}\binom{n}{k}(-1)^{k-1}H_{k,2}=&\frac{H_n}{n}.
\end{align}
\end{lemma}

\begin{remark}\label{Rem3.1} The identities (\ref{equa1})-(\ref{equa3}) are known.
See \cite[(1.45)]{Gould} for \eqref{equa1}, \cite{H} for \eqref{equa2},
and the proof of \cite[Lemma 3.1]{S15} for simple proofs of \eqref{equa2} and \eqref{equa3}.
\end{remark}

Let $p>3$ be a prime.
In 1900 J.W.L. Glaisher \cite{Glaisher-1,Glaisher-2} refined Wolstenholme's work
\cite{Wolstenholme} on congruences
by showing that
\begin{align}\label{harmo1}
H_{p-1}\equiv -\frac{p^2}{3}B_{p-3}\pmod {p^3}
\end{align}
and
\begin{equation}\label{harmo2}
H_{p-1,s}\equiv\frac{s}{s+1}pB_{p-1-s}\pmod {p^2}\ \ \t{for all}\ s=1,\ldots,p-2.
\end{equation}

\begin{lemma}\label{Lem3.2} {\rm (X. Zhou and T. Cai \cite[p.\,1332]{ZC})} Let $r,s\in\Z^+$. For any prime $p>rs+2$, we have
\begin{align*}\sum_{1\ls i_1<i_2<\ldots<i_r\ls p-1}\f1{i_1^si_2^s\ldots i_r^s}
\eq\begin{cases}(-1)^r\f{s(rs+1)}{2(rs+2)}p^2B_{p-rs-2}\pmod {p^3}&\t{if}\ 2\nmid rs,
\\(-1)^{r-1}\f{s}{rs+1}pB_{p-rs-1}\pmod{p^2}&\t{if}\ 2\mid rs.\end{cases}
\end{align*}
\end{lemma}

\begin{lemma}\label{Lem3.3} {\rm (i) (J. Zhao \cite[Theorems 3.1 and 3.2]{Zh})} Let $s,t\in\Z^+$ and let $p\gs s+t$ be a prime. Then
$$\sum_{1\ls j<k\ls p-1}\f1{j^sk^t}\eq\f{(-1)^t}{s+t}\bi{s+t}sB_{p-s-t}\pmod p.$$
If $s\eq t\pmod2$ and $p>s+t+1$, then
\begin{align*}\sum_{1\ls j<k\ls p-1}\f1{j^sk^t}\eq&\l((-1)^st\bi{s+t+1}s-(-1)^ts\bi{s+t+1}t-s-t\r)
\\&\times\f{pB_{p-s-t-1}}{2(s+t+1)}\pmod{p^2}.
\end{align*}

{\rm (ii) (J. Zhao \cite[Theorem 3.5]{Zh})} Let $r,s,t\in\Z^+$ with $r+s+t$ odd. For any prime $p>r+s+t$, we have
\begin{align*}\sum_{1\ls i<j<k\ls p-1}\f1{i^rj^sk^t}\eq&\l((-1)^r\bi{r+s+t}r-(-1)^t\bi{r+s+t}t\r)
\\&\times\f{B_{p-(r+s+t)}}{2(r+s+t)}\pmod p.
\end{align*}

{\rm (iii) (J. Zhao \cite[Corollary 3.6]{Zh})} For any prime $p>5$, we have
$$\sum_{1\ls i<j\ls k\ls p-1}\f1{i^2jk}\eq\sum_{1\ls i<j\ls k\ls p-1}\f1{ij^2k}
\eq\sum_{1\ls i<j\ls k\ls p-1}\f1{ijk^2}\eq0\pmod p.$$
\end{lemma}

\begin{lemma}\label{Lem3.4} For any prime $p>5$, we have
\begin{equation}\label{ij}\sum_{k=1}^{p-1}\sum_{1\ls i<j\ls k-1}\f1{i^2j^2}\eq \f25pB_{p-5}-3\f{H_{p-1}}{p^2}\pmod{p^2}
\end{equation}
and
\begin{equation}\label{hij}\sum_{k=1}^{p-1}\sum_{1\ls i<j\ls k-1}\f{H_{k-1}}{i^2j^2}\eq\f{3H_{p-1}}{p^2}\pmod p.
\end{equation}
\end{lemma}
\Proof. (i) Clearly,
\begin{align*}&\sum_{k=1}^{p-1}\sum_{1\ls i<j\ls k-1}\f1{i^2j^2}+\sum_{1\ls i<j\ls p-1}\f1{i^2j^2}
\\=&\sum_{1\ls i<j\ls p-1}\f{\sum_{j<k\ls p}1}{i^2j^2}=\sum_{1\ls i<j\ls p-1}\f{p-j}{i^2j^2}
\end{align*}
and hence
\begin{equation}\label{i2j2}\sum_{k=1}^{p-1}\sum_{1\ls i<j\ls k-1}\f1{i^2j^2}=(p-1)\sum_{1\ls i<j\ls p-1}\f1{i^2j^2}-\sum_{1\ls i<j\ls p-1}\f1{i^2j}.
\end{equation}

Note that $H_{p-1,3}\eq0\pmod{p^2}$ by \eqref{harmo2} or Lemma \ref{Lem3.2},
and
\begin{equation}\label{tau}\sum_{k=1}^{p-1}\f{H_{k-1}}{k^2}\eq-3\f{H_{p-1}}{p^2}\pmod{p^2}
\end{equation}
by R. Tauraso \cite[Theorem 1.3]{Tauraso}. Thus
\begin{align*}\sum_{1\ls i<j\ls p-1}\f1{i^2j}=&\sum_{i=1}^{p-1}\f{H_{p-1}-H_i}{i^2}
\\=&H_{p-1}H_{p-1,2}-H_{p-1,3}-\sum_{i=1}^{p-1}\f{H_{i-1}}{i^2}\eq3\f{H_{p-1}}{p^2}\pmod{p^2}.
\end{align*}
By Lemma 3.2,
\begin{equation}\label{22}\sum_{1\ls i<j\ls p-1}\f1{i^2j^2}\eq-\f25pB_{p-5}\pmod{p^2}.
\end{equation}
Combining these with \eqref{i2j2}, we immediately obtain
\eqref{ij}.

(ii) As $H_{p-1}\eq0\pmod p$ and
$$\sum_{s=1}^{p-1}H_s=\sum_{1\ls r\ls s\ls p-1}\f1r=\sum_{r=1}^{p-1}\f{p-r}r=pH_{p-1}-(p-1)\eq1\pmod p,$$
we have
\begin{align*}&\sum_{k=1}^{p-1}\sum_{1\ls i<j\ls k-1}\f{H_{k-1}}{i^2j^2}
\\\eq&\sum_{1\ls i<j\ls p-1}\f{\sum_{s=1}^{p-1}H_s-\sum_{0<s<j}H_s}{i^2j^2}
\\\eq&\sum_{1\ls i<j\ls p-1}\f1{i^2j^2}\l(1-\sum_{1\ls r\ls s<j}\f1r\r)
\\=&\sum_{1\ls i<j\ls p-1}\f1{i^2j^2}\l(1-\sum_{0<r<j}\f{j-r}r\r)
=\sum_{1\ls i<j\ls p-1}\f1{i^2j}\(1-\sum_{0<r<j}\f1r\)
\\=&\sum_{1\ls i<j\ls p-1}\f1{i^2j}-\sum_{1\ls i<j\ls p-1}\f1{i^3j}-\sum_{1\ls r<i<j\ls p-1}\f1{ri^2j}-\sum_{1\ls i<r<j\ls p-1}\f1{i^2rj}
\\\eq&\f{(-1)^1}3\bi 32B_{p-3}-\f{(-1)^1}4\bi43B_{p-4}-0-0\eq\f{3H_{p-1}}{p^2}\pmod p
\end{align*}
with the help of Lemma \ref{Lem3.3} and \eqref{harmo1}. This proves \eqref{hij}.

The proof of Lemma \ref{Lem3.4} is now complete. \qed

\begin{lemma}\label{Lem3.5} Let $p>5$ be a prime. Then
\begin{equation}\label{910}\sum_{k=1}^{p-1}\bi pk\f{(-1)^k}kH_{k-1,2}\eq \f {9}{10} p^2B_{p-5}\pmod{p^3}.\end{equation}
\end{lemma}
\Proof. As
$$(-1)^{k-1}\bi{p-1}{k-1}=\prod_{0\ls j<k}\l(1-\f pj\r)\eq1-pH_{k-1}\pmod {p^2}$$ for all $k=1,\ldots,p-1$, we have
\begin{align*}\sum_{k=1}^{p-1}\bi pk\f{(-1)^k}kH_{k-1,2}=&p\sum_{k=1}^{p-1}\bi{p-1}{k-1}\f{(-1)^k}{k^2}H_{k-1,2}
\\\eq& p\sum_{k=1}^{p-1}\f{pH_{k-1}-1}{k^2}H_{k-1,2}\pmod{p^3}.
\end{align*}
By \eqref{22},
$$\sum_{k=1}^{p-1}\f{H_{k-1,2}}{k^2}\eq -\f25pB_{p-5}\pmod{p^2}.$$
Note also that
\begin{align*}\sum_{k=1}^{p-1}\f{H_{k-1}H_{k-1,2}}{k^2}
=&\sum_{1\ls j<k\ls p-1\atop 1\ls i<k}\f1{ij^2k^2}
\\=&\sum_{1\ls j<k\ls p-1}\f1{j^3k^2}+\sum_{1\ls i<j<k\ls p-1}\f1{ij^2k^2}+\sum_{1\ls j<i<k\ls p-1}\f1{j^2ik^2}
\\\eq&\f{(-1)^2}{3+2}\bi{3+2}3B_{p-5}+\l((-1)^1\bi 51-(-1)^2\bi 52\r)\f{B_{p-5}}{2\times5}+0
\\\eq&2B_{p-5}-\f{3}2B_{p-5}=\f {B_{p-5}}2\pmod p
\end{align*}
with the help of parts (i) and (iii) of Lemma \ref{Lem3.3}. Combining the above, we get the desired congruence
\eqref{910}. \qed

\begin{lemma}\label{Lem3.6} For any prime $p>5$, we have
$$\bi{2p-1}{p-1}\eq1+2pH_{p-1}-\f 45p^5B_{p-5}\pmod{p^6}.$$
\end{lemma}
\Proof. It is known (cf. \cite[Theorem 2.4]{Tauraso}) that
$$\binom{2p-1}{p-1}\equiv 1+2pH_{p-1}+\frac{2}{3}p^3H_{p-1,3}\pmod{p^6}.$$
By Lemma \ref{Lem3.2},
\begin{equation}\label{h3}H_{p-1,3}\eq-\f{3\times4}{2\times5}p^2B_{p-5}=-\f 65p^2B_{p-5}\pmod{p^3}.
\end{equation}
So, the desired congruence follows. \qed

\begin{lemma}\label{Lem3.7} For any prime $p>7$, we have
\begin{equation}\label{236}H_{p-1}+\f p2H_{p-1,2}+\f{p^2}6H_{p-1,3}\eq0\pmod{p^6}.
\end{equation}
\end{lemma}
\Proof.  Clearly,
\begin{align*}\sum_{k=1}^{p-1}\f1{p-k}
\eq&\sum_{k=1}^{p-1}\f{p^6-k^6}{p-k}\l(-\f1{k^6}\r)
=-\sum_{k=1}^{p-1}\f{p^5+p^4k+p^3k^2+p^2k^3+pk^4+k^5}{k^6}
\\=&-p^5H_{p-1,6}-p^4H_{p-1,5}-p^3H_{p-1,4}-p^2H_{p-1,3}-pH_{p-1,2}-H_{p-1}
\\\eq&-p^3H_{p-1,4}-p^2H_{p-1,3}-pH_{p-1,2}-H_{p-1}\pmod{p^6}
\end{align*}
since $H_{p-1,6}\eq0\pmod p$ and $H_{p-1,5}\eq0\pmod{p^2}$ by \eqref{harmo2} or Lemma \ref{Lem3.2}.
Note that
$$H_{p-1,4}\eq 4p\l(\f{B_{2p-6}}{2p-6}-2\f{B_{p-5}}{p-5}\r)\eq-\f2{3p}H_{p-1,3}\pmod{p^3}$$
by \cite[Theorem 5.1 and Remark 5.1]{SZH}.
Therefore
$$2H_{p-1}+pH_{p-1,2}\eq \f23p^2H_{p-1,3}-p^2H_{p-1,3}\pmod{p^6}$$
and hence \eqref{236} follows. \qed

\begin{remark}\label{Rem3.2}
Let $p>3$ be a prime. If $p>7$ then \eqref{236} has the following equivalent form
\begin{equation}\label{H12}
2H_{p-1}+pH_{p-1,2}\eq \f25p^4B_{p-5}\pmod{p^5}
\end{equation}
in view of \eqref{h3}.
It is easy to see that \eqref{H12} also holds for $p=5,7$.
\end{remark}

\section{Proof of Theorem \ref{main1}}
\setcounter{equation}{0}

To prove Theorem \ref{main1}, we need an auxiliary theorem.

\begin{theorem}\label{three}
Let $p>5$ be a prime. Then
\begin{align}
&\label{cong1}\quad \sum_{k=1}^{p-1}\binom{p}{k}^2\binom{p+k-1}{k-1}\equiv pH_{p-1}+\f{9}{10}p^5B_{p-5}\pmod{p^6},\\
&\label{cong2}\quad \sum_{k=1}^{p-1}\binom{p}{k}^2\binom{p+k-1}{k-1}k\equiv -3p^2H_{p-1}+ \f{21}{10}p^6B_{p-5}\pmod{p^7},
\end{align}
and
\begin{equation}\label{cong3}\begin{aligned}
&\quad \sum_{k=1}^{p-1}\binom{p}{k}^2\binom{p+k-1}{k-1}k^2
\\\equiv& -p^2-(4p^5+4p^4+2p^3)H_{p-1}+\f45p^7B_{p-5}\pmod{p^8}.
\end{aligned}\end{equation}
\end{theorem}
\begin{proof} For each $k=1,\ldots,p-1$, we clearly have
\begin{equation}\label{pk}\begin{aligned}\f1p\bi pk=&\f{(-1)^{k-1}}k\prod_{0< j<k}\l(1-\f pj\r)
\\\eq&\f{(-1)^{k-1}}k\(1-pH_{k-1}+p^2\sum_{1\ls i<j\ls k-1}\f1{ij}\)\pmod{p^3}
\end{aligned}\end{equation}
and
\begin{equation}\label{pkk}\begin{aligned}&(-1)^{k-1}\bi{p-1}{k-1}\bi{p+k-1}{k-1}
\\=&\prod_{0<j<k}\left(1-\f{p^2}{j^2}\right)
\eq 1-p^2H_{k-1,2}+p^4\sum_{1\ls i<j\ls k-1}\f1{i^2j^2}\pmod{p^6}.
\end{aligned}
\end{equation}

Let $r\in\{0,1,2\}$ and
\begin{equation}\label{sigma}\sigma_r:=\sum_{k=1}^{p-1}\bi pk^2\bi{p+k-1}{k-1}k^r.
\end{equation}
In view of the above, we have
\begin{align*}
\sigma_r=&p\sum_{k=1}^{p-1}\bi pkk^{r-1}\bi{p-1}{k-1}\bi{p+k-1}{k-1}
\\\eq&p\sum_{k=1}^{p-1}\bi pkk^{r-1}(-1)^{k-1}\(1-p^2H_{k-1,2}+p^4\sum_{1\ls i<j\ls k-1}\f1{i^2j^2}\)
\\\eq&p\sum_{k=1}^{p-1}\bi pkk^{r-1}(-1)^{k-1}+p^3\sum_{k=1}^{p-1}\bi pkk^{r-1}(-1)^kH_{k-1,2}
\\&+p^5\sum_{k=1}^{p-1}\f pk(1-pH_{k-1})k^{r-1}\sum_{1\ls i<j\ls k-1}\f1{i^2j^2}\pmod {p^8}.
\end{align*}

 {\it Case} 1. $r=0$.

 In this case,
 $$p\sum_{k=1}^{p-1}\bi pkk^{r-1}(-1)^{k-1}=pH_p-1=pH_{p-1}$$
 by \eqref{equa1}, and hence
$$\sigma_0\eq pH_{p-1}+p^3\sum_{k=1}^{p-1}\bi pk\f{(-1)^k}kH_{k-1,2}
\eq pH_{p-1}+\f 9{10}p^5B_{p-5}\pmod{p^6}$$
with the help of Lemma \ref{Lem3.5}. This proves \eqref{cong1}.

 {\it Case} 2. $r=1$.

 In this case,
 $$p\sum_{k=1}^{p-1}\bi pkk^{r-1}(-1)^{k-1}=-p\sum_{k=0}^p\bi pk(-1)^k=-p(1-1)^p=0,$$
 and also
 $$\sum_{1\ls i<j<k\ls p-1}\f1{i^2j^2k}\eq\l((-1)^2\bi 52-(-1)^1\bi 51\r)\f{B_{p-5}}{2\times5}=\f32B_{p-5}\pmod{p}$$
 by Lemma \ref{Lem3.3}.
 Thus
 \begin{equation}\label{sigma1}\sigma_1
 \eq p^3\sum_{k=1}^{p-1}\bi pk(-1)^kH_{k-1,2}+\f 32p^6B_{p-5}\pmod{p^7}.
 \end{equation}

 In view of \eqref{equa3} and \eqref{pk}, we have
 \begin{align*}\sum_{k=1}^{p-1}\bi pk(-1)^kH_{k-1,2}
 =&\sum_{k=1}^p\bi pk(-1)^kH_{k,2}+H_{p,2}-\sum_{k=1}^{p-1}\bi pk\f{(-1)^k}{k^2}
 \\\eq&-\f{H_p}p+H_{p,2}+p\sum_{k=1}^{p-1}\f{1-pH_{k-1}+p^2\sum_{1\ls i<j\ls k-1}\f1{ij}}{k^3}
 \\\eq&-\f{H_{p-1}}p+H_{p-1,2}+pH_{p-1,3}-p^2\sum_{1\ls j<k\ls p-1}\f1{jk^3}
 \\&+p^3\sum_{1\ls i<j<k\ls p-1}\f1{ijk^3}\pmod{p^4}.
 \end{align*}
 Note that
 $$H_{p-1,2}\eq\f25p^3B_{p-5}-\f 2pH_{p-1}\pmod{p^4} \ \t{and} \ H_{p-1,3}\eq-\f 65p^2B_{p-5}\pmod{p^3}$$
  by \eqref{H12} and \eqref{h3}.
  In view of Lemma \ref{Lem3.3},
 $$\sum_{1\ls j<k\ls p-1}\f1{jk^3}\eq\l(-3\bi 51-(-1)\bi 53-1-3\r)\f{pB_{p-5}}{2\times 5}=-\f 9{10}pB_{p-5}\pmod{p^2}$$
 and
 $$\sum_{1\ls i<j<k\ls p-1}\f1{ijk^3}\eq\l((-1)^1\bi 51-(-1)^3\bi 53\r)\f{B_{p-5}}{2\times5}=\f{B_{p-5}}2\pmod p.$$
 Therefore
 \begin{align*}\sum_{k=1}^{p-1}\bi pk(-1)^kH_{k-1,2}\eq&-3\f{H_{p-1}}p+\f25p^3B_{p-5}-\f 65p^3B_{p-5}+\f 9{10}p^3B_{p-5}+\f{p^3}2B_{p-5}\\\eq&-\f 3pH_{p-1}+\f{3}5p^3B_{p-5}\pmod{p^4}.
 \end{align*}
 Combining this with \eqref{sigma1}, we see that
 $$\sigma_1\eq p^3\l(-\f 3pH_{p-1}+\f{3}5p^3B_{p-5}\r)+\f 32p^6B_{p-5}=-3p^2H_{p-1}+\f{21}{10}p^6B_{p-5}\pmod{p^7}.$$
 This proves \eqref{cong2}.

 {\it Case} 3. $r=2$.

 In this case,
 $$p\sum_{k=1}^{p-1}\bi pkk^{r-1}(-1)^{k-1}=p^2\sum_{k=1}^{p-1}\bi{p-1}{k-1}(-1)^{k-1}=p^2\l((1-1)^{p-1}-1\r)=-p^2$$
 and hence
 \begin{align*}
 \sigma_2\eq&-p^2+p^4\sum_{k=1}^{p-1}\bi{p-1}{k-1}(-1)^kH_{k-1,2}
 \\&+p^6\sum_{k=1}^{p-1}\sum_{1\ls i<j\ls k-1}\f1{i^2j^2}-p^7\sum_{k=1}^{p-1}\sum_{1\ls i<j\ls k-1}\f{H_{k-1}}{i^2j^2}
 \\\eq&-p^2+p^4\l(\f{H_{p-1}}{p-1}+H_{p-1,2}\r)+p^6\l(\f 25pB_{p-5}-3\f{H_{p-1}}{p^2}\r)-p^7\f{3H_{p-1}}{p^2}
 \\\eq&-p^2+p^4\l(-(p+1)H_{p-1}+\f25p^3B_{p-5}-\f2pH_{p-1}\r)
 \\&+\f25p^7B_{p-5}-3p^4H_{p-1}-3p^5H_{p-1}\pmod{p^8}
 \end{align*}
 with the help of \eqref{equa3}, \eqref{H12} and Lemma \ref{Lem3.4}.
 This yields the desired \eqref{cong3}.

 In view of the above, we have completed the proof of Theorem \ref{three}. \end{proof}
\medskip
\noindent{\it Proof of Theorem \ref{main1}}. It is easy to see that \eqref{1.4}-\eqref{1.6} hold for $p=5$. Below we assume $p>5$.

For $i=1,2,3$ let
$$
S_i:=\sum_{k=0}^{p-1}\binom{p}{k+1}^2\binom{p+k}{k}a_{i}(p,k)
$$
with $a_i(n, k)$ given by Lemma \ref{Lem2.1}.
In view of Lemma \ref{Lem2.1}, it suffices to show the
following three congruences:
\begin{align}\label{C1}
S_1&\equiv -4p^2-4p^7B_{p-5}\pmod{p^8},\\
\label{C2}S_2&\equiv 2p^2+10p^3H_{p-1}+p^7B_{p-5}\pmod{p^8},
\\\label{CC}S_3&\equiv 2p^3-3p^2+(10p^4-5p^3)H_{p-1}-\f 32p^7B_{p-5}\pmod {p^8},
\end{align}
Clearly,
\begin{equation}\label{Si}\begin{aligned}
S_{i}=&\sum_{j=1}^{p}\binom{p}{j}^2\binom{p+j-1}{j-1}a_{i}(p,j-1)\\
=&\binom{2p-1}{p-1}a_{i}(p,p-1)+\sum_{k=1}^{p-1}\binom{p}{k}^2\binom{p+k-1}{k-1}a_{i}(p,k-1).
\end{aligned}\end{equation}
Note that
$$a_1(p,p-1)=-p^3-3p^2,\ a_2(p,p-1)=2p^2\ \t{and}\ a_3(p,p-1)=p^3-2p^2.$$
In view of Lemma \ref{Lem3.6}, we have
\begin{equation}\label{2pp}\begin{aligned}&\bi{2p-1}{p-1}a_i(p,p-1)
\\\eq&\l(1+2pH_{p-1}-\f 45p^5B_{p-5}\r)a_i(p,p-1)
\\\eq&\begin{cases}-p^3-2p^4H_{p-1}-3p^2-6p^3H_{p-1}+\f{12}5p^7B_{p-5}\pmod{p^8}&\t{if}\ i=1,
\\2p^2+4p^3H_{p-1}-\f 85p^7B_{p-5}\pmod{p^8}&\t{if}\ i=2,
\\p^3+2p^4H_{p-1}-2p^2-4p^3H_{p-1}+\f 85p^7B_{p-5}\pmod{p^8}&\t{if}\ i=3.\end{cases}
\end{aligned}\end{equation}

For $r=0,1,2$ let $\sigma_r$ be defined as in \eqref{sigma}.
Noting $$a_1(p,k-1)=(p^3-p^2)-(p^2+3p)k-(p-1)k^2$$
and applying Theorem \ref{three}, we get
\begin{align*}
&\sum_{k=1}^{p-1}\bi pk^2\bi{p+k-1}{k-1}a_1(p,k-1)
\\=&(p^3-p^2)\sigma_0-(p^2+3p)\sigma_1-(p-1)\sigma_2
\\\eq&(p^3-p^2)\l(pH_{p-1}+\f{9}{10}p^5B_{p-5}\r)
-(p^2+3p)\l(-3p^2H_{p-1}+\f{21}{10}p^6B_{p-5}\r)
\\&-(p-1)\l(-p^2-(4p^5+4p^4+2p^3)H_{p-1}+\f 45p^7B_{p-5}\r)
\\\eq&p^3-p^2+(2p^4+6p^3)H_{p-1}-\f{32}5p^7B_{p-5}\pmod{p^8}.
\end{align*}
Combining this with \eqref{Si} and \eqref{2pp}, we get \eqref{C1}.

Similarly,
\begin{align*}
S_2\eq&2p^2+4p^3H_{p-1}-\f 85p^7B_{p-5}+3p^2\sigma_0-p\sigma_1
\\\eq&2p^2+4p^3H_{p-1}-\f 85p^7B_{p-5}+3p^2\l(pH_{p-1}+\f{9}{10}p^5B_{p-5}\r)
\\&-p\l(-3p^2H_{p-1}+\f{21}{10}p^6B_{p-5}\r)
\\\eq&2p^2+10p^3H_{p-1}-p^7B_{p-5}\pmod{p^8}.
\end{align*}
This proves \eqref{C2}.

As
$$a_3(p,k-1)=4p^3-2p^2-(2p^2+p)k-(p-1)k^2,$$
we have
\begin{align*}
&\sum_{k=1}^{p-1}\bi pk^2\bi{p+k-1}{k-1}a_3(p,k-1)
\\=&(4p^3-2p^2)\sigma_0-(2p^2+p)\sigma_1-(p-1)\sigma_2
\\\eq&(4p^3-2p^2)\l(pH_{p-1}+\f{9}{10}p^5B_{p-5}\r)
-(2p^2+p)\l(-3p^2H_{p-1}+\f{21}{10}p^6B_{p-5}\r)
\\&-(p-1)\l(-p^2-(4p^5+4p^4+2p^3)H_{p-1}+\f 45p^7B_{p-5}\r)
\\\eq&p^3-p^2+(8p^4-p^3)H_{p-1}-\f{31}{10}p^7B_{p-5}\pmod{p^8}.
\end{align*}
Combining this with \eqref{Si} and \eqref{2pp}, we get \eqref{CC}.

The proof of Theorem \ref{main1} is now complete. \qed

\section{Proofs of theorems \ref{a16b4} and \ref {a32b8}}\label{bandA}

\setcounter{theorem}{0}
\setcounter{equation}{0}

Clearly $L_2(0)=0$, and for each $n\in\Z^+$ we have
$$L_2(n)=L_2\l(\l\lfloor\f n2\r\rfloor\r)+\begin{cases}0&\t{if}\ n\eq0,3\pmod4,\\1&\t{if}\ n\eq1,2\pmod4.
\end{cases}$$
E.\,Rowland and R.\,Yassawi discovered that the two least
 significant binary digits of ~$L_2(n)$
are completely determined by the fourth and the third
least significant digits of  the Ap\'ery number $A_n$.
(As $A_n\eq1\pmod4$, the two least significant digits of $A_n$
are always $0$ and $1$, and hence these digits contain no information
about~$n$.)

\begin{theorem} \label{Th5.1}\cite[Theorem 3.28]{RY}
Let $n\in\N$. If
\begin{equation}
  A_n\equiv 8j+4i+1 \pmod{ 16}
\end{equation} with $\{i,j\}\subseteq \{0,1\}$,
then
\begin{equation}
L_2(n)\equiv 2j+i \pmod{4}.
\end{equation}\label{RY}
\end{theorem}

Theorem \ref{a16b4} shows that
 the fourth and the third
least significant digits of~$A_n$
also determine the second least significant digit
of $\beta_n$ (all numbers $\beta_n$ are  odd, so the very least significant digit
is always~$1$).

Theorem \ref{a32b8} shows that
 the knowledge of the five least significant digits of $A_n$
by itself is not sufficient for
 determining the  three least significant
digits of~$L_2(n)$. This is because the third least significant digit of
$\beta_n$ can be $0$ or $1$. But the additional knowledge of this digit
gives the missing bit of information
for determining the three  least significant digits of
$L_2(n)$ according to Theorem~\ref{a32b8}.

\begin{figure}[ht]
 \center
 \includegraphics[width=.9\textwidth]{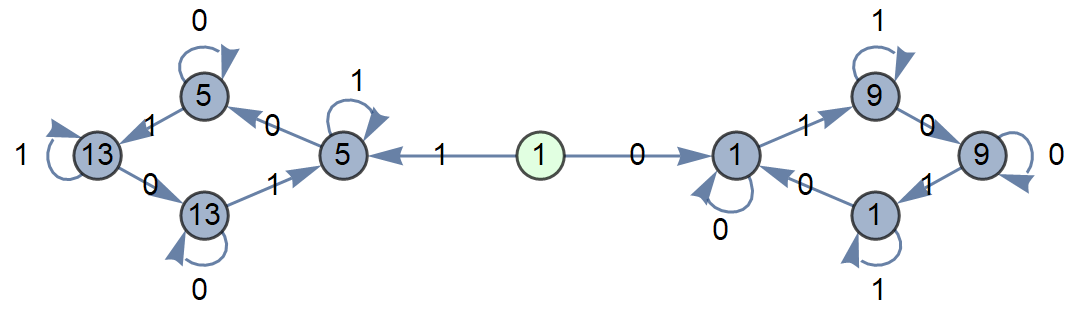}
 \caption{Automaton $\mathcal {A}_{16}$
 calculating  $A_n\!\!\mod{16}$}
 \label{A16}
 \end{figure}

\begin{figure}[ht]
 \center
 \includegraphics[width=.5\textwidth]{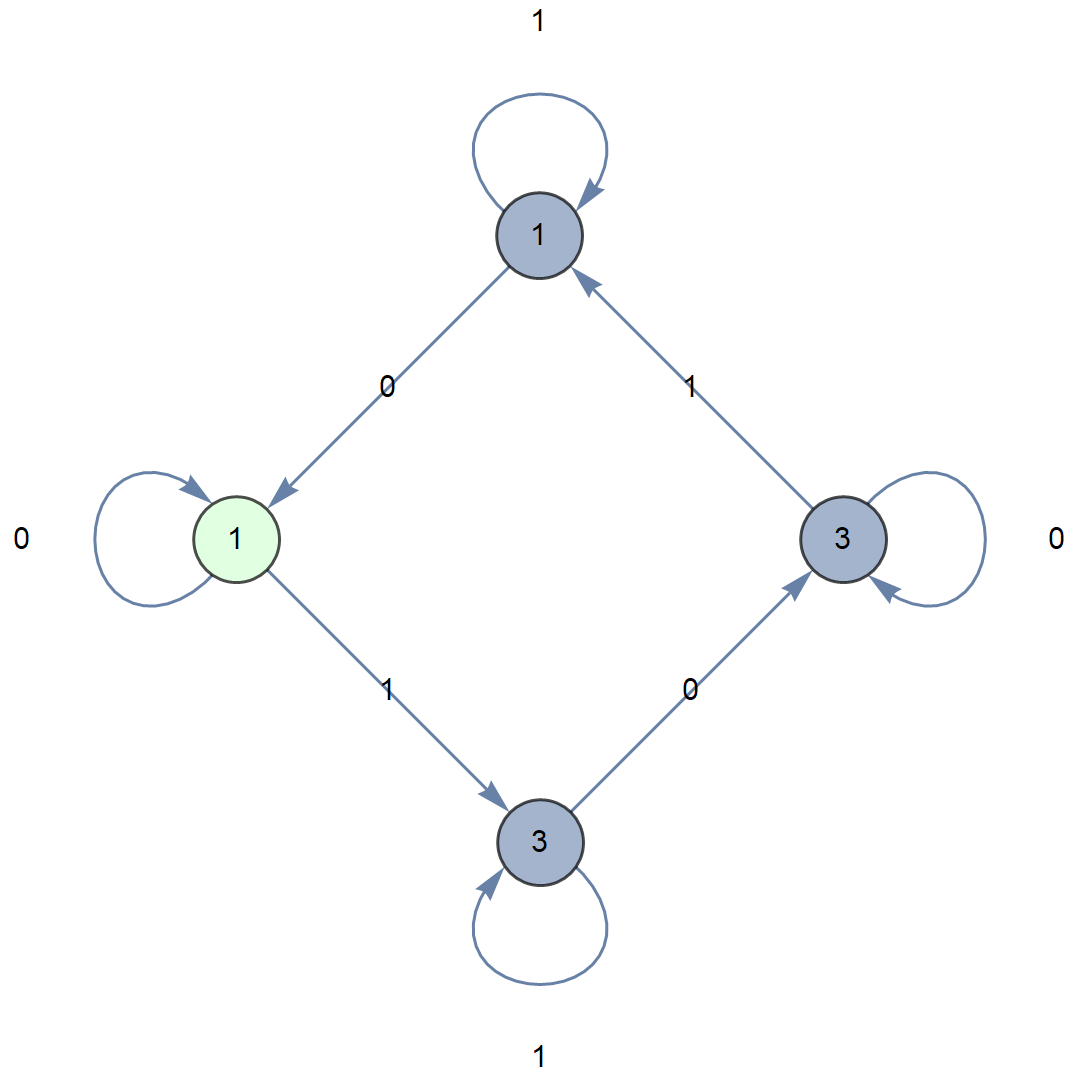}
 \caption{Automaton $\mathcal {B}_{4}$ calculating  $\beta_n\!\!\mod{4}$}
 \label{B4}
 \end{figure}

Theorems \ref{a16b4} and \ref{a32b8} can be proved
by the technique used in \cite{RY}
 with the aid of
software \cite{Rsoft} implemented by E.\,Rowland.

Namely,
 Ap\'{e}ry numbers
of both kinds
are known to be the diagonal sequences of certain rational functions.
In particular, A.\,Straub \cite{Straub} proved that $A_n$ is the
coefficient of $x_1^nx_2^nx_3^nx_4^n$ in the formal
Taylor expansion  of the function
\begin{equation}
\frac{1}{
 (1 - x_1 - x_2 )(1 - x_3 - x_4 ) - x_1 x_2 x_3 x_4}.
\end{equation}
Similar coefficient in the expansion of
\begin{equation}\label{polyb}
\frac{1}{
 (1 - x_1 )(1 - x_2 )(1 - x_3 )(1 -x_4 ) - (1 - x_1 )x_1 x_2 x_3}
\end{equation}
is, according to \cite{RY}, equal to $\beta_n$.

According to \cite[Theorem 2.1]{RY},
the above representations of Ap\'ery numbers   as the diagonal coefficients
imply that for every prime $p$ and its power $q$
the sequences  $A_n\!\!\mod{q}$ and $\beta_n\!\!\mod{q}$
are $p$-automatic. Moreover,
\cite{RY} contains two algorithms  for
constructing corresponding automata;
{\tt Mathematica} implementations of these algorithms are given
in \cite{Rsoft}.

To prove Theorem \ref{a16b4}, we need to examine the sequence
 $A_n\!\!\mod{16}$. Corresponding automaton $\mathcal {A}_{16}$
 was calculated in \cite{Rpaper} and exhibited in \cite{RY}, we reproduce it here
 in Figure~\ref{A16} (with the initial state marked in light color). This automaton
 calculates  $A_n\!\!\mod{16}$ in the following way. Let
 \begin{equation}
   n=\sum_{k=0}^m n_k2^k
 \end{equation}
where $\{n_0,\dots,n_m\}\subseteq\{0,1\}$;
start from the initial state and follow the oriented path
formed by edges marked $n_0,\dots,n_m$; the final vertex is labeled by
$A_n\!\!\mod{16}$.

Function  {\tt AutomaticSequenceReduce} from \cite{Rsoft}, being applied to  \eqref{polyb},
returns the minimal
automaton $\mathcal{ B}_4$ calculating
$\beta_n\!\!\mod{4}$; this automaton is exhibited in Figure~\ref{B4}.

Let
\begin{equation*}
  f(1)=f(13)=1\quad\t{and}\quad f(5)=f(9)=3.
\end{equation*}
Theorem \ref{a16b4} asserts that
\begin{equation}\label{th}
  (\beta_n\!\!\!\mod{4})=f(A_n\!\!\!\mod{16}).
\end{equation}
We can construct an automaton calculating $f(A_n\!\!\mod{16})$
simply by applying function $f$ to the labels of states in
automaton $\mathcal {A}_{16}$. The resulting automaton is
not minimal, but we can minimize it using,
for example,  function
{\tt AutomatonMinimize} from \cite{Rsoft}. The minimal automaton calculating $f(A_n\!\!\mod{16})$
turns out to be equal to~$\mathcal{ B}_4$, which proves \eqref{th}
and hence Theorem \ref{a16b4} is proved as well.

\begin{figure}[ht]
 \center
 \includegraphics[width=.9\textwidth]{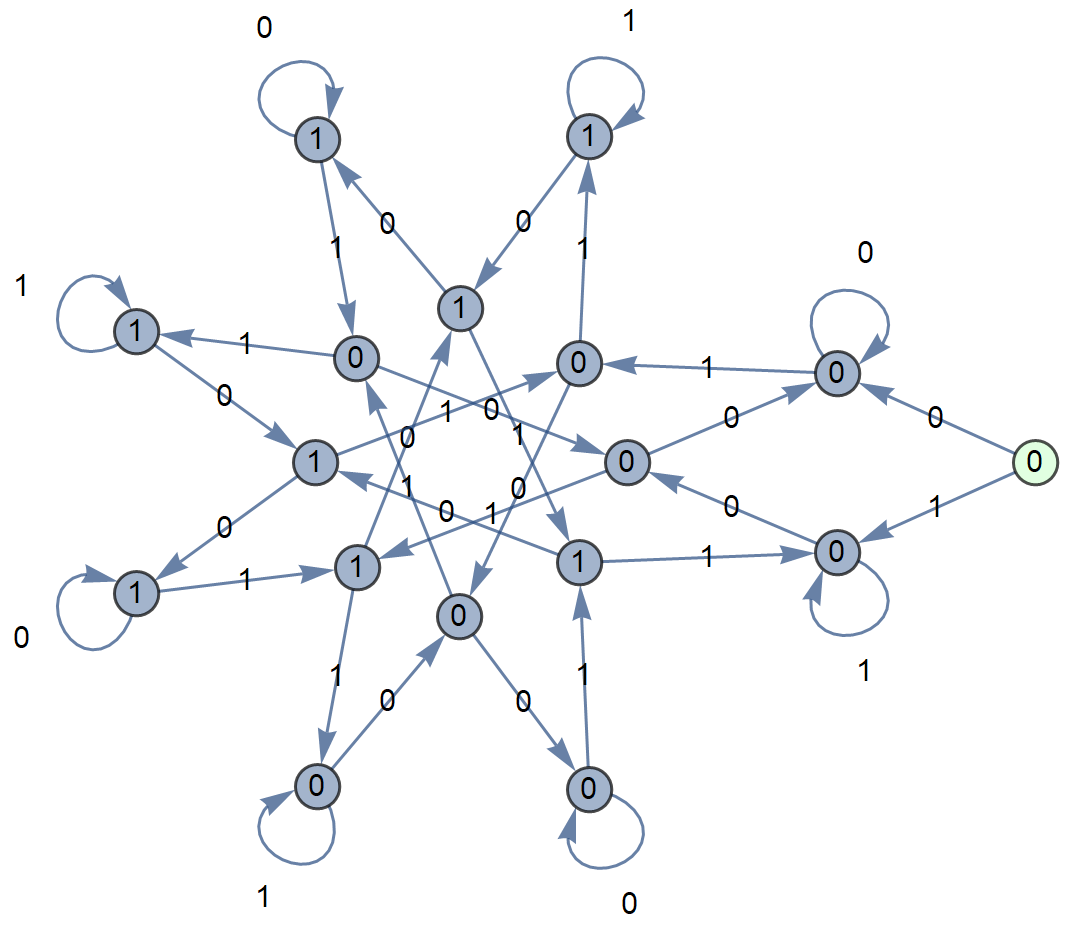}
 \caption{Automaton $\mathcal {A}_{32,5}$
 calculating  the fifth digit of $A_n$}
 \label{A325}
 \end{figure}

\begin{figure}[ht]
 \center
 \includegraphics[width=.5\textwidth]{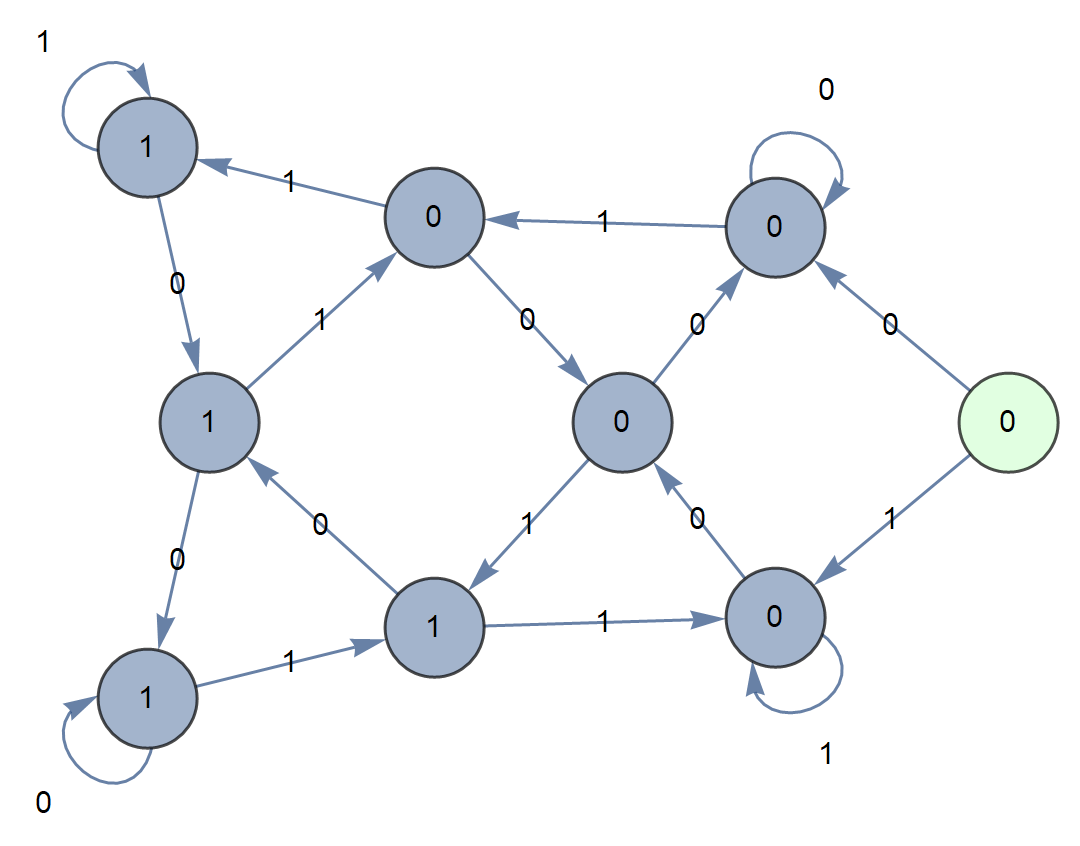}
 \caption{Automaton $\mathcal {B}_{8,3}$ calculating
 the third digit of $\beta_n$}
 \label{B83}
 \end{figure}

\begin{figure}[ht]
 \center
 \includegraphics[width=.5\textwidth]{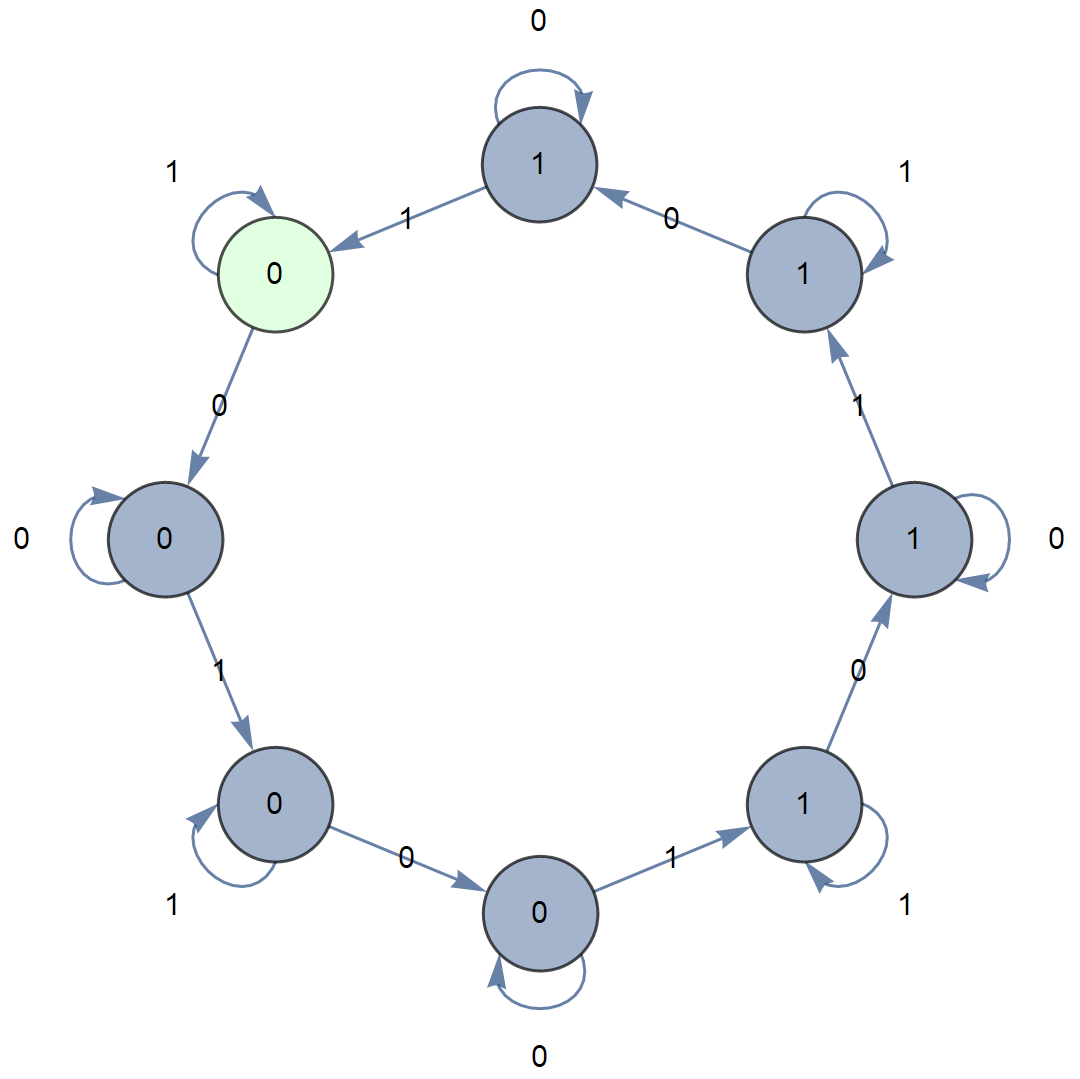}
 \caption{Automaton  calculating
 the third digit of $L_2(n)$}
 \label{C3}
 \end{figure}

The minimal automaton $\mathcal {A}_{32}$
calculating  $A_n\!\!\mod{32}$ was constructed in \cite{Rpaper}, it has 33 states;  the minimal
automaton $\mathcal {B}_{8}$ calculating  $\beta_n\!\!\mod{8}$ has 17 states.
Taking into account Theorem~\ref{Th5.1},
for proving Theorem \ref{a32b8} it is sufficient to work
just with the fifth least significant digits of $A_n$
and the third least significant digits of
$\beta_n$ and $L_2(n)$, which allows us to deal with automata
having  fewer
number of states.

Figures~\ref{A325} and \ref{B83} exhibit the minimal
automata $\mathcal {A}_{32,5}$ and $\mathcal {B}_{8,3}$
calculating  respectively the fifth least significant digit of $A_n$
and  the third least significant digit of
$\beta_n$.
These automata are easily constructed from
$\mathcal {A}_{32}$ and $\mathcal {B}_{8}$
by properly relabelling the states
and minimizing resulting automata.

With function {\tt AutomatonProduct} from \cite{Rsoft} we
can easily construct automaton $\mathcal {C}$ which is
 the product of
$\mathcal {A}_{32,5}$ and $\mathcal {B}_{8,3}$. Each state  of  $\mathcal {C}$
is labeled by two bits  (corresponding to the fifth digit of $A_n$ and
the third digit of $\beta_n$). Replacing each such pair of bits by their
sum modulo $2$ and performing minimization,
we get the automaton exhibited on Figure~\ref{C3}. It is
not difficult  to see
that  this automaton calculates  the third digit of~$L_2(n)$,
which proves Theorem~\ref{a32b8}.

\Ack. The authors would like to thank the third author's PhD student Chen Wang
for helpful comments on Lemmas \ref{Lem2.1} and \ref{Lem3.7}.

\end{document}